\documentclass[12pt,reqno]{article}
\usepackage{amsmath,amsthm,amsfonts,amssymb,amscd}
\usepackage[dvips]{graphicx}
\usepackage{psfrag}

\theoremstyle{plain}
\newtheorem{thm}{Theorem}[section]
\newtheorem{rk}[thm]{Remark}
\newtheorem{prop}[thm]{Proposition}

\newtheorem{lemma}[thm]{Lemma}
\newtheorem{defi}[thm]{Definition}

\newtheorem{maintheorem}{Theorem}

\newcommand{\re}{{\Bbb R}}

\newcommand{\e}{\epsilon}
\newcommand{\de}{\delta}

\title{Sectional-Anosov flows in higher dimensions}
\author{A. M. L\'opez B.
        \thanks{
This work is partially supported by CAPES, Brazil.}}
\date{}

\begin{document}
\maketitle

\begin{abstract}
A {\em sectional-Anosov flow} on a manifold $M$ is a $C^1$ vector field 
inwardly transverse to the boundary for which the maximal invariant is 
sectional-hyperbolic \cite{mo}. 
We prove that every attractor of every vector field $C^1$ close to 
a transitive sectional-Anosov flow with singularities 
on a compact manifold has a singularity. 
This extends the three-dimensional result obtained in \cite{m}.
\end{abstract}

%%==================================================================SECTION 1======================================================================

\section{Introduction}

\noindent
The {\em sectional-Anosov flows} were introduced in \cite{mo} as a generalization of the {\em Anosov flows} including
also the {\em saddle-type hyperbolic attracting sets}, the {\em geometric and multidimensional Lorenz attractors} \cite{abs}, \cite{bpv}, \cite{gw}.
Some properties of these flows have been proved elsewhere in the literature \cite{ap}, \cite{lec}.
In particular, \cite{m} proved that every attractor of every vector field $C^1$ close to a transitive sectional-Anosov flow with singularities on a compact $3$-manifold has a singularity.
A generalization of this result in \cite{ams} asserts that
every attractor of every vector field $C^1$ close to a nonwandering sectional-Anosov flow with singularities of a compact $3$-manifold has a singularity.
In this paper we give a further generalization but to higher dimensions. More precisely, we prove that every attractor of every vector field $C^1$ close to
a transitive sectional-Anosov flow with singularities of a compact manifold has a singularity. We do not prove the result for nonwandering sectional-Anosov flows
due to the lack of certain three-dimensional results like, for instance, the improved sectional-Anosov connecting lemma (compare with \cite{ams}).
Let us state our result in a precise way.

Consider a compact manifold $M$ of dimension $n\geq 3$ (a {\em compact $n$-manifold} for short).
We denote by $\partial M$ the boundary of $M$.
Let ${\cal X}^1(M)$ be the space
of $C^1$ vector fields in $M$ endowed with the
$C^1$ topology.
Fix $X\in {\cal X}^1(M)$, inwardly
transverse to the boundary $\partial M$ and denotes
by $X_t$ the flow of $X$, $t\in I\!\! R$.\\

The $\omega$-limit set of $p\in M$ is the set
$\omega_X(p)$ formed by those $q\in M$ such that $q=\lim_{n\infty}X_{t_n}(p)$ for some
sequence $t_n\to\infty$.
We have $\omega_X(p)\subset \Omega(X)$ for every $p\in M$. \\

Let $\Lambda$ be an compact invariant set of $X$,
i.e. $X_t(\Lambda)=\Lambda$ for all $t\in I\!\! R$.
We say that $\Lambda$ is {\em transitive} if
$\Lambda=\omega_X(p)$ for some $p\in \Lambda$.
We say that $\Lambda$ is {\em singular} if it
contains a singularity of $X$.
We say that $\Lambda$ is
{\em attracting}
if $\Lambda=\cap_{t>0}X_t(U)$
for some compact neighborhood $U$ of $\Lambda$.
This neighborhood is called
{\em isolating block} of $\Lambda$.
It is well known that the isolating block $U$ can be chosen to be
positively invariant, namely $X_t(U)\subset U$ for all
$t>0$.
An {\em attractor} is a transitive attracting set.
An attractor is {\em nontrivial} if it is
not a closed orbit.\\

The {\em maximal invariant} set of $X$ is defined by 
$M(X)= \bigcap_{t \geq 0} X_t(M)$.\\

We denote by $m(L)$ the minimum norm of a linear
operator $L$, i.e., $m(L)= inf_{v \neq 0} \frac{\left\|Lv\right\|}{\left\|v\right\|} $.\\

\begin{defi}
\label{d2}
A compact invariant set
$\Lambda$ of $X$
is {\em partially hyperbolic}
if there is a continuous invariant
splitting
$$
T_\Lambda M=E^s\oplus E^c
$$
such that the following properties
hold for some positive constants $C,\lambda$:

\begin{enumerate}
\item
$E^s$ is {\em contracting}, i.e.
$$
\mid\mid DX_t(x) \left|_{E^s_x}\right. \mid\mid
\leq Ce^{-\lambda t},
$$
for all $x\in \Lambda$ and $t>0$.
\item
$E^s$ {\em dominates} $E^c$, i.e.
$$
\frac{\mid\mid DX_t(x) \left|_{E^s_x}\right. \mid\mid}{m(DX_t(x) \left|_{E^c_x}\right. )}
\leq Ce^{-\lambda t},
$$
for all $x\in \Lambda$ and $t>0$.
\end{enumerate}

We say the central subbundle $E^c_x$ of $\Lambda$ is 
{\em sectionally-expanding} if
$$dim(E^c_x) \geq 2 \quad 
and \quad 
\left| det(DX_t(x) \left|_{L_x}\right. ) \right| \geq C^{-1}e^{\lambda t}, \qquad
\forall x \in \Lambda \quad and \quad t > 0  $$ 
for all two-dimensional subspace $L_x$ of $E^c_x$.
Here $det(DX_t(x) \left|_{L_x}\right. )$ denotes
the jacobian of $DX_t(x)$ along $L_x$.
\end{defi}
\bigskip

\begin{defi}
\label{shs}
A {\em sectional-hyperbolic set}
is a partially hyperbolic set whose singularities (if any) are hyperbolic and whose central subbundle is sectionally-expanding.
\end{defi}

Recall that a singularity of a vector field is hyperbolic if
the eigenvalues of its linear part
have non zero real part.\\

\begin{defi}
\label{secflow}
We say that $X$ is a {\em sectional-Anosov flow} if $M(X)$ is a sectional-hyperbolic
set.
\end{defi}

Our result is the following.

\begin{maintheorem}
\label{thB}
Let $X$ be a transitive sectional-Anosov flow with singularities
of a compact $n$-manifold.
Then, every attractor of every vector field $C^1$ close to $X$ has a singularity.
\end{maintheorem}

The proof follows closely that in \cite{m}, namely, we assume by contradiction that there is a sequence
$X^n$ of vector fields converging to $X$ each one with a non-singular attractor $A^n$.
As in \cite{m} we shall prove both that the sequence $A^n$ accumulates on a singularity $\sigma$ of $X$
and that such accumulation do imply that $A^n$ contains a singularity for $n$ large.
This required to extend some preliminary lemmas in \cite{m},  using a new definition of Lorenz-like 
singularity for sectional Anosov flows and a new definition for singular cross section for 
Lorenz-like singularities to the higher-dimensional case.

%%============================================================================SECTION 2====================================================

\section{Lorenz-like singularities and singular cross-sections in higher dimension}

\noindent
Let $M$ be a compact $n$-manifold, $n \geq 3$.
Fix $X\in {\cal X}^1(M)$, inwardly
transverse to the boundary $\partial M$. We denotes
by $X_t$ the flow of $X$, $t\in I\!\! R$, and $M(X)$ the maximal invariant of $X$.\\

\begin{defi}
\label{hyperbolic}
A compact invariant set $\Lambda$ of $X$ is {\em
hyperbolic}
if there are a continuous tangent bundle
invariant decomposition
$T_{\Lambda}M=E^s\oplus E^X\oplus E^u$ and positive constants
$C,\lambda$ such that

\begin{itemize}
\item $E^X$ is the vector field's
direction over $\Lambda$.
\item $E^s$ is {\em contracting}, i.e.,
$
\mid\mid DX_t(x) \left|_{E^s_x}\right.\mid\mid
\leq Ce^{-\lambda t}$, 
for all $x \in \Lambda$ and $t>0$.
\item $E^u$ is {\em expanding}, i.e.,
$
\mid\mid DX_{-t}(x) \left|_{E^u_x}\right.\mid\mid
\leq Ce^{-\lambda t},
$
for all $x\in \Lambda$ and $t> 0$.
\end{itemize}
A closed orbit is hyperbolic if it is also hyperbolic, as a compact invariant set. An attractor is hyperbolic if it is also a hyperbolic
set. 
\end{defi}

It follows from the stable manifold theory \cite{hps} that if $p$ belongs to a hyperbolic set $\Lambda$, then the following sets

\begin{tabular}{lll}
$W^{ss}_X(p)$ & = & $\{x:d(X_t(x),X_t(p))\to 0, t\to \infty\},$ \\
$W^{uu}_X(p)$ & = & $\{x:d(X_t(x),X_t(p))\to 0, t\to -\infty\},$ \\
%$W^{ss}_X(p,\epsilon)$ & = & $\{x:d(X_t(x),X_t(p))\leq\epsilon, \forall t\geq 0\},$ and, \\
%$W^{uu}_X(p,\epsilon)$ & = & $\{x:d(X_t(x),X_t(p))\leq \epsilon, \forall t\leq 0\}.$ \\
\end{tabular}\\

are $C^1$ immersed submanifolds of $M$ which are tangent at $p$ to the subspaces $E^s_p$ and $E^u_p$ of $T_pM$ respectively.
Similarly, the set

\begin{tabular}{lll}
$W^{s}_X(p)$ & = & $ \bigcup_{t\in I\!\! R}W^{ss}_X(X_t(p))$, \\
$W^{u}_X(p)$ & = & $ \bigcup_{t\in I\!\! R}W^{uu}_X(X_t(p))$. \\
\end{tabular}\\

are also $C^1$ immersed submanifolds tangent to $E^s_p\oplus E^X_p$ and $E^X_p\oplus E^u_p$ at $p$ respectively.
Moreover, for every $\epsilon>0$ we have that

\begin{tabular}{lll}
$W^{ss}_X(p,\epsilon)$ & = & $\{x:d(X_t(x),X_t(p))\leq\epsilon, \forall t\geq 0\},$ and, \\
$W^{uu}_X(p,\epsilon)$ & = & $\{x:d(X_t(x),X_t(p))\leq \epsilon, \forall t\leq 0\}$\\
\end{tabular}\\

are closed neighborhoods of $p$ in $W^{ss}_X(p)$ and $W^{uu}_X(p)$ respectively.

There is also a stable manifold theorem in the case when $X$ is sectional-Anosov.
Indeed, denoting by $T_{M(X)}M=E^s_{M(X)}\oplus E^c_{M(X)}$ the corresponding the sectional-hyperbolic
splitting over $M(X)$ we have from \cite{hps} that
the contracting subbundle $E^s_{M(X)}$
can be extended to a contracting subbundle $E^s_U$ in $M$. Moreover, such an extension is tangent to a continuous foliation denoted by $W^{ss}$ (or $W^{ss}_X$ to indicate dependence on $X$).
By adding the flow direction to $W^{ss}$ we obtain a continuous foliation $W^s$ (or $W^s_X$) now tangent to $E^s_M\oplus E^X_M$.
Unlike the Anosov case $W^s$ may have singularities, all of which being
the leaves $W^{ss}(\sigma)$ passing through the singularities $\sigma$ of $X$.
Note that $W^s$ is transverse to $\partial M$
because it contains the flow direction (which is transverse to $\partial M$
by definition).

It turns out that every singularity $\sigma$ of a sectional-Anosov flow $X$ satisfies $W^{ss}_X(\sigma)\subset W^s_X(\sigma)$.
Furthermore, there are two possibilities for such a singularity, namely,
either $dim(W^{ss}_X(\sigma))=dim(W^s_X(\sigma))$ (and so $W^{ss}_X(\sigma)=W^s_X(\sigma)$) or $dim(W^{s}_X(\sigma))=dim(W^{ss}_X(\sigma))+1$.
In the later case we call it Lorenz-like according to the following definition.

\begin{defi}
\label{ll} 
We say that a singularity $\sigma$ of a sectional-Anosov flow $X$ is {\em Lorenz-like}
if
$
dim (W^s(\sigma))
=
dim (W^{ss}(\sigma))+1.
$
\end{defi}

Let $\sigma$ be a singularity Lorenz-like of a sectional-Anosov flow $X$. 
We will denote $dim(W^{ss}_X(\sigma))=s$ and 
$dim(W^{u}_X(\sigma))=u$, 
therefore $\sigma$ has a $(s+1)$-dimensional 
local stable manifold $W^s_X(\sigma)$. 
Moreover $W^{ss}_X(\sigma)$ separates $W^s_{loc}(\sigma)$
in two connected components denoted by $W^{s,t}_{loc}(\sigma)$ 
and $W^{s,b}_{loc}(\sigma)$ respectively.

\begin{defi}
\label{n1}
A {\em singular-cross section} of a Lorenz-like singularity 
$\sigma$ will be a pair of submanifolds $\Sigma^t, \Sigma^b$, where 
$\Sigma^t, \Sigma^b$ are cross sections and;\\

\begin{tabular}{l}
$\Sigma^t$ is transversal to $W^{s,t}_{loc}(\sigma)$. \\
$\Sigma^b$ is transversal to $W^{s,b}_{loc}(\sigma)$.\\ 
\end{tabular}\\

Note that every singular-cross section
contains a pair singular submanifolds $l^t,l^b$ 
defined as the intersection of the local stable manifold of $\sigma$ with $\Sigma^t,\Sigma^b$ 
respectively.
Also note that $dim(l^*)=dim(W^{ss}(\sigma))$.\\

If $*=t,b$ then $\Sigma^*$ is a {\em hipercube of dimension $(n-1)$},   
i.e., diffeomorphic to $B^u[0,1] \times B^{ss}[0,1]$, with 
$B^u[0,1] \approx I^u$, $B^{ss}[0,1] \approx I^s$, $I^k=[-1,1]^k$, $k \in \mathbb{Z}$ and  where:\\

\begin{tabular}{l}
$B^{ss}[0,1]$ is a ball centered at zero and radius $1$ contained in $\re^{dim(W^{ss}(\sigma))}=\re^{s}$\\
$B^{u}[0,1]$ is a ball centered at zero and radius $1$ contained in $\re^{dim(W^{u}(\sigma))}=\re^{n-s-1}$
\end{tabular}\\

Let $f: B^u[0,1] \times B^{ss}[0,1] \longrightarrow \Sigma^*$ be the diffeomorphism, 
where $$f(\left\{0\right\} \times B^{ss}[0,1])=l^*$$ and $\left\{0\right\}=0 \in \re^u$.
Hence, we denoted the boundary of $\Sigma^*$ for $\partial \Sigma^*$,
and $\partial \Sigma^* = \partial^h \Sigma^* \cup \partial^v \Sigma^*$ such that\\

\begin{tabular}{l}
$\partial^h \Sigma^* = \left\{\right.$the union of the boundary submanifolds
 which are transverse to $l^*$ $\left.\right\}$\\
$\partial^v \Sigma^* = \left\{\right.$ the union of the
boundary submanifolds which are parallel to $l^*$ $\left.\right\}.$
\end{tabular}\\

Moreover,
$$\partial^h \Sigma^*  =  (I^u \times [\cup_{j=0}^{s-1} (I^j \times \left\{-1\right\} \times I^{s-j-1})])
\bigcup (I^u \times[\cup_{j=0}^{s-1} (I^j \times \left\{1\right\} \times I^{s-j-1})])$$
%%Fixa a imagem ("eixo y", parte estável forte) e viaja pelo bordo da base ("eixo x" parte central)
$$\partial^v \Sigma^*  =  ([\cup_{j=0}^{u-1} (I^j \times \left\{-1\right\} \times I^{u-j-1})] \times I^s)
\bigcup ([\cup_{j=0}^{u-1} (I^j \times \left\{1\right\} \times I^{u-j-1})] \times I^s)$$
%%Fixa a imagem ("eixo y", parte estável forte) e viaja pelo bordo da base ("eixo x" parte central)	
and where $I^0 \times I=I$.\\

\end{defi}

Hereafter we denote $\Sigma^* = B^u[0,1] \times B^{ss}[0,1]$.

%%============================================================================SECTION 3====================================================

\section{Sectional hyperbolic sets in higher dimension}

In this section we use some definitions and results for higher dimension
and we extend some preliminary results for 
transitive sectional-Anosov flows.

An useful property of sectional-hyperbolic sets
is given below.

\begin{lemma}
\label{l1}
Let $X$ be a sectional-Anosov flow,
$X$ a $C^1$ vector field 
in $M$.
If $Y$ is $C^1$ close to $X$,
then every nonempty, compact, non singular, invariant set $H$
of $Y$ is
hyperbolic {\em saddle-type} (i.e. $E^s\neq 0$ and $E^u\neq 0$).
\end{lemma}
\begin{proof}
See (\cite{mpp2}).
The proof in \cite{mpp2} is made in dimension three, but the same proof yields the same conclusion
in any dimension.
\end{proof}

\begin{lemma}
\label{l1'}
Let $X$ be a transitive sectional-Anosov flow
$C^1$ in $M$.
If $O\subset M(X)$ is a periodic orbit
of $X$,
then $O$ is a hyperbolic saddle-type periodic orbit.
In addition, if $p\in O$ then
the set
$$
\{q\in W^{uu}_X(p):
M(X)=\omega_X(q)\}
$$
is dense in $W^{uu}_X(p)$.
\end{lemma}

\begin{proof}
(See \cite{m})
\end{proof}

This following theorem appears in \cite{lec}. First we examine the sectional-hyperbolic 
splitting $T_{M(X)}M = E^s_{M(X)} \oplus E^c_{M(X)}$
of a sectional hyperbolic set $M(X)$ of $X \in {\cal X}^1(M)$.

Lorenz-like singularities
are considered below.

\begin{thm}
\label{t1}
Let $X$ be a transitive sectional-Anosov flow 
$C^1$ for $M$.
Then,
every $\sigma\in Sing(X)\cap M(X)$
is Lorenz-like
and satisfies
$$
M(X)\cap W^{ss}_X(\sigma)=\{\sigma\}.
$$
\end{thm}
\begin{proof}
  
We make proofs the following claims for the theorem:\\

{\em Claim 1:}\\ 
If $x \in (M(X) \setminus Sing(X))$, then $X(x) \notin E^s_x$.
\begin{proof} 
Suppose by contradiction that there is $x_0 \in (M(X) \setminus Sing(X))$ such that 
$X(x_0)\in E^s_{x_0}$ . Then, $X(x) \in E^s_x$ 
for every $x$ in the orbit of $x_0$ since $E^s_{M(X)}$
is invariant. So $X(x)\in E^s_x$ 
for every $x \in \alpha(x_0)$ by continuity. It follows that $\omega(x)$ 
is a singularity for all
$x \in \alpha(x_0)$. In particular, $\alpha(x_0)$ contains a singularity 
$\sigma$ which is necessary saddle-type.
Now we have two cases: $\alpha(x_0)=\{\sigma\}$ or not. If $\alpha(x_0)=\{\sigma\}$ 
then $x_0 \in W^u(\sigma)$.
For all $t \in \mathbb{R}$ define the unitary vector
$$v^t = \frac{DX_t (x_0)(X(x_0))}{ || DX_t (x_0)(X(x_0)) ||}.$$
It follows that 
$$v^t \in T_{X_t (x_0)}W^u(\sigma)\cap	E^s_{X_t(x_0)}, \quad \forall t \in \mathbb{R}$$.
Take a sequence $t_n \rightarrow \infty$ such that the sequence $v^{-tn}$ 
converges to $v^{\infty}$ (say). Clearly
$v^{\infty}$ is an unitary vector and, since $X_{-t_n} (x_0)\rightarrow  \sigma$ 
and $E^s$ is continuous we obtain
$$v^{\infty} \in T_{\sigma}W^u(\sigma)\cap E^s_{\sigma}$$.

Therefore $v^{\infty}$ is an unitary vector which is simultaneously expanded and contracted
by $DX_t(\sigma)$ a contradiction. This contradiction shows the result when $\alpha(x_0) = \{\sigma\}$.
If $\alpha(x_0) \neq \{\sigma\}$ then 
$(W^u(\sigma) \setminus \{\sigma\}) \cap \alpha(x_0) \neq \emptyset$. 
Pick $x_1 \in (W^u(\sigma)\setminus \{\sigma\})\cap \alpha(x_0)$.
Clearly $X(x_1) \in E^s_{x_1}$ 
and then we get a contradiction as in the first case replacing $x_0$
by $x_1$. This contradiction proves the lemma in the second case.
\end{proof}

{\em Claim 2:}\\
If $\sigma \in M(X) \cap Sing(X)$, then $M(X) \cap W^{ss}(\sigma) = \{\sigma\}$.

\begin{proof} 
Notice that $E^s_x = T_xW^{ss}(\sigma)$ for all $x \in W^{ss}(\sigma)$. 
Moreover, $W_{ss}(\sigma)$ is an invariant
manifold so $X(x) \in T_xW^{ss}(\sigma)$ for all $x \in W^{ss}(\sigma)$.
We conclude that $X(x) \in E^s_x$
for all $x \in W^{ss}(\sigma)$ and now Claim (1) applies.
\end{proof}

\end{proof}

This theorem implies the following
two useful properties.

\begin{prop}
\label{Ladilla1}
Let $X$ be a transitive sectional-Anosov flow 
$C^1$ of $M$.
Let $\sigma$ be a singularity of $X$ in $M(X)$
(so $\sigma$ is Lorenz-like by Theorem
\ref{t1}).
Then, there is a
singular-cross section
$\Sigma^t,\Sigma^b$ of $\sigma$ in $M$
such that
$$
\left(M(Y)\right)\cap (\partial^h\Sigma^t\cup\partial^h\Sigma^b)=
\emptyset,
$$
for every $C^r$ vector field $Y$
close to $X$.
\end{prop}
\begin{proof}
See (\cite{m}).
\end{proof}

\bigskip

Let $\sigma$ be a Lorenz-like singularity
of a $C^1$ vector field
$X$ in ${\cal X}^1(M)$, and $\Sigma^t,\Sigma^b$ be a singular-cross
section of $\sigma$. Thus for $\sigma$ we denote,
\begin{equation}
\begin{tabular}{l}
$dim(W^{ss}_X(\sigma))=s$, then \\
$dim(W^s_X(\sigma))=s+1$ and $dim(W^u_X(\sigma))=n-s-1$, \\ 
$dim(\Sigma^*)=s+(n-s-1)=n-1$.\\
%%$dim(\Sigma^*)+dim(W^s_X(\sigma))=n+s$.\\
\end{tabular}\\
\label{p1}
\end{equation}                     

We remember that $\Sigma^*= B^u[0,1] \times B^{ss}[0,1]$, 
then we will set up a family of singular cross-sections as follows:
Given $0<\Delta\leq 1$ small, we define $\Sigma^{*,\Delta}=B^u[0,\Delta] \times B^{ss}[0,1]$,
such that
$$l^* \subset \Sigma^{*,\Delta}\subset \Sigma^* \,\,\, i.e$$
$$(l^*=\{0\}\times B^{ss}[0,1]) \subset
(\Sigma^{*,\Delta}=B^u[0,\Delta] \times B^{ss}[0,1]) \subset 
(\Sigma^*= B^u[0,1] \times B^{ss}[0,1])$$
where fix a coordinate system $(x^*,y^*)$
in $\Sigma^*$ and $*=t,b$. We will assume that $\Sigma^*=\Sigma^{*,1}$.

We will be use this notation in
the next lemma and in the next section for the Theorem \ref{thB}. Thus,

\begin{lemma}
\label{Ladilla3}
Let $X$ be a transitive sectional-Anosov flow
$C^1$ of $M$. Let $\sigma$ be a
singularity of $X$ in $M(X)$.
Let $Y^n$ be a sequence of vector fields
converging to $X$ in the $C^1$ topology.
Let $O_n$ be a periodic
orbit of $Y^n$ such that the
sequence $\{O_n:n\in I\!\! N\}$ accumulates on $\sigma$.
If $0<\Delta\leq 1$ and
$\Sigma^t,\Sigma^b$ is a singular-cross section
of $\sigma$, then there is
$n$ such that
either
$$
O_n\cap int(\Sigma^{t,\Delta})\neq\emptyset\,\,\,\,
\mbox{or}
\,\,\,\,O_n\cap int(\Sigma^{b,\Delta})\neq\emptyset.
$$
\end{lemma}

\begin{proof}
Since $O_n$ accumulates on $\sigma\in M(X)$
and $M(X)$ is maximal invariant, 
we have that $O_n\subset M(X)$ for all $n$ large
(recall $Y^n\to X$ as $n\to \infty$).
Let us fix a fundamental domain
$D_{\e}$ of the vector field's flow $X_t$ restricted
to the local stable manifold
$W^s_{loc}(\sigma)$ (\cite{dmp}) for $\e>0$ as follows:\\

\begin{tabular}{lll}
$D_{\e}$  & = & $S_{\e} \cup S_{-\e} \cup C_{\e}$, where: \\
$S_{\e}$  & = & $\{ x \in \re^{s+1} | \qquad \Sigma_{i=1}^{s} x_i^2 + (x_{s+1}-\e)^2 = 1,\qquad \wedge \qquad x_{s+1} \geq \e \}$ \\
$S_{-\e}$ & = & $\{ x \in \re^{s+1} | \qquad \Sigma_{i=1}^{s} x_i^2 + (x_{s+1}-\e)^2 = 1,\qquad \wedge \qquad x_{s+1} \leq -\e \}$ \\
$C_{\e}$  & = & $\{ x \in \re^{s+1} | \qquad \Sigma_{i=1}^{s} x_i^2 = 1, \qquad \wedge \qquad x_{s+1} \in [-\e,\e] \}$ 
\end{tabular}\\

%%===============================================================FIGURE 2=============================================================================
\begin{figure}[htv] 
\begin{center}
\input{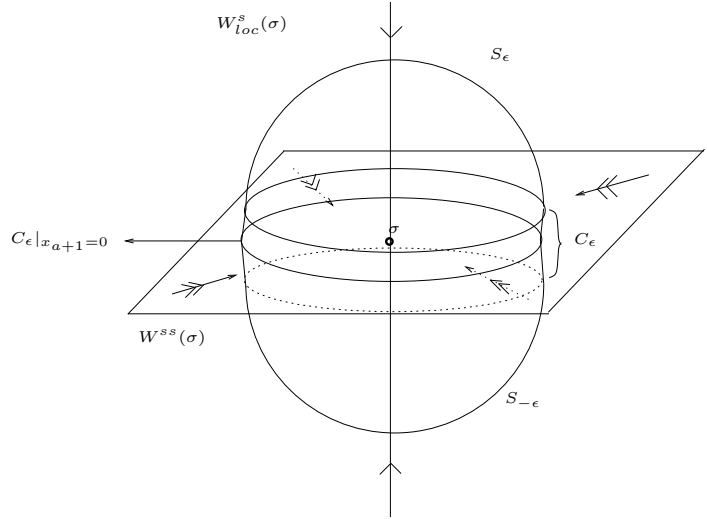} 
\caption{\label{f.2} The fundamental domain.}
\end{center}
\end{figure}
%%======================================================================================================================================================

Since $W^s_{loc}(\sigma)$ is $(s+1)$-dimensional
and $D_{\e}$ is homeomorphic to the sphere $(s+1)$-dimensional, 
one has that for construction $D_{\e}$ intersects
$W^{ss}_X(\sigma)$ in $C_{\e}|_{x_{s+1}=0}$ that is a sphere $(s-1)$-dimensional.
Note that the orbits of all point in $C_{\e}|_{x_{s+1}=0}$ together with
$\sigma$ yields $W^{ss}_X(\sigma)$.
In particular, $C_{\e}|_{x_{s+1}=0} \notin M(X)$ by Theorem
\ref{t1}. Also note that forall $\e$, $D_{\e}$ is a fundamental domain.

Let $\tilde{D_{\e}}$ be a cross section
of $X$ such that
$W^s_{loc}(\sigma)\cap \tilde{D_{\e}}=D_{\e}$.
It follows that $\tilde{D_{\e}}$ is a $(s+2)$-cylinder, and so,
we can put a system coordinated
$(x,s)$ with $x \in D_{\e}$ and
$s\in [-1,1]$ say, and so we can construct a family of singular-cross
sections $\Sigma^t_\delta,\Sigma^b_\delta$
(for all $\delta\in [-\e,\e]$)
by setting
$$\Sigma^t_\delta=\{(x,s)\in \tilde{D_{\e}}: x \in S_{\de} , s \in [-1,1]\}$$
and
$$\Sigma^b_\delta=\{(x,s)\in \tilde{D_{\e}}: x \in S_{-\de} , s \in [-1,1]\}$$

Due to the smooth variation
of $W_Y^{ss}(\sigma(Y))$ with respect to
$Y$ close to $X$ we can assume that
$\sigma(Y)=\sigma$ and that
$W^{ss}_{loc,Y}(\sigma(Y)
=W^{ss}_{loc}(\sigma)$
for every $Y$ close to $X$.
By choosing $D_{\e}$ so close to $\sigma$ we can further assume that
$\tilde{D_{\e}}$ is a cross-section
of $Y$, for every $Y$ close to $X$.
We claim that there is $\delta>0$ such that
the conclusion of the lemma
holds for
$\Sigma^t=\Sigma^t_\delta$ and $\Sigma^b=\Sigma^b_\delta$.
Indeed, we first note that under the cylindrical coordinate
system $(x,s)$ one has $\Sigma^{*,\Delta}=\Sigma^*_\Delta$ for all
$0<\Delta\leq \delta$ (where $*=t,b$). 
If the conclusion of the claim fails, 
implies that $O_n$ intersects
$\tilde{D_{\e}}\setminus
(\Sigma^t_\Delta\cup \Sigma^b_\Delta)$ for all $\Delta>0$ small.
In other words,
there would exist $p_n\in O_n$ (for all $n$ large)
such that
$p_n=(x_n,s_n)$ with $x_n\in C_{\Delta}$ and $s_n\to 0$ as $n\to\infty$.
Since $\Delta$ is arbitrary and $s_n\to 0$ we conclude
that $p_n$ converges to a point in $C_{\Delta}|_{x_{s+1}=0}$ 
by passing to a subsequence if necessary,
since if $s_n\to 0$, it implies that the intersection 
tends to $(s+1)$-dimensional sphere $D_{\e}$.

As $O_n\subset M(Y^n)$, $Y^n\to X$ and $M(Y^n)$ is $\epsilon C^1$-close to
$M(X)$ for all $n$ $(n\in \mathbb{N})$, we have that the last would imply
that exists a point $z \in (C_{\e}|_{x_{s+1}=0})$ such that $ z \in M(X)$.
This contradicts
Theorem \ref{t1} and the proof follows.
\end{proof}
\bigskip

%%=======================================================================SECTION 4============================================================

\section{Proof of Theorem \ref{thB}}

We prove the theorem by contradiction.
Let $X$ be a transitive sectional-Anosov flow
$C^1$ of $M$.
Then, we suppose that
there is a sequence $X^n\overset{C^1}{\to} X$ such that every $X^n$ exhibits
a non-singular attractor $A^n \in M(X^n)$ arbitrarily close to $M(X)$ 
and since $A^n$ also is arbitrarily close to $M(X)$, 
we can assume that
$A^n$ belongs to $M(X)$ for all $n$.
Also, since $A^n$ is an attractor, 
we have that $A^n$ is compact, invariant and nonempty, 
and by hypothesis $A^n$ is non-singular, then 
the lemmas \ref{l1} and \ref{l1'} imply
the following:
\begin{equation}
\begin{tabular}{l}
$A^n$ is a hyperbolic attractor of type saddle
of $X^n$ for all $n$,\\ 
and since $A^n$ is non-singular for all $n$, obviously $A^n$ is not\\ 
a singularity of $X^n$ for all $n$.
\end{tabular}
\label{e1}
\end{equation}
\bigskip

We denote by:\\ 
\begin{tabular}{l}
$Sing(X)$ the set of singularities of $X$.\\
$Cl(A)$ the closure of $A$, $A\subset M$.\\
If $\delta>0$, $B_\delta(A)=\{x\in M:d(x,A)<\delta\}$, 
where $d(\cdot,\cdot)$ is the metric in $M$.\\
\end{tabular}\\

As in \cite{m}, we need the following lemma in
the higher dimension case.

\begin{lemma}
\label{Ladilla4}
The attractors $A^n$ accumulate 
on $Sing(X)$, i.e.
$$Sing(X)\bigcap Cl\left(\bigcup_{n\in I\!\! N}A^n\right)\neq\emptyset.$$
\end{lemma}

\begin{proof}
We prove the lemma by contradiction. Then,
we suppose that there is
$\delta>0$, such that

\begin{equation}
\label{*}
B_\delta(Sing(X))\bigcap\left(\bigcup_{n\in I\!\! N}A^n\right)=\emptyset.
\end{equation}

In the same way as in \cite{m}, we define
$$H=\bigcap_{t\in I\!\! R}X_t\left(M\setminus B_{\delta/2}(Sing(X))\right).$$
Obviously $Sing(X)\cap H=\emptyset$.

Note that $H$ is
compact since $M$ is.
It follows that $H$ is a nonempty compact
set \cite{m}, which is clearly invariant for $X$.
It follows that $H$
is hyperbolic by Lemma \ref{l1}
since $Sing(X)\cap H=\emptyset$.
Denote by $E^s\oplus E^X\oplus E^u$ the corresponding hyperbolic splitting (see Definition \ref{hyperbolic}).

By the stability of hyperbolic sets we can fix a neighborhood $W$ of
$H$ and $\epsilon>0$ such that
if $Y$ is a vector field
$C^r$ close to $X$
and $H_Y$ is a compact invariant set
of $Y$ in $W$ then :
\begin{equation}
\begin{tabular}{l}\\
$H_Y$ is hyperbolic and its
hyperbolic splitting
$E^{s,Y}\oplus E^Y\oplus
E^{u,Y}$.\\

$dim(E^u)=dim(E^{u,Y})$, 
$dim(E^s)=dim(E^{s,Y})$.\\

The manifolds
$W^{uu}_Y(x,\epsilon)$, $x\in H_Y$, 
have uniform size $\epsilon$. \\
\end{tabular}
\label{e2}
\end{equation}

As $X^n\to X$, we have that:
\begin{equation}
\begin{tabular}{l}\\
$\bigcap_{t\in I\!\! R}X^n_t(M\setminus B_{\delta/2}(Sing(X))\subset W$, 
for all $n$ large.\\

$A^n\subset M\setminus B_{\delta/2}(Sing(X)) $
for all $n$, and $A^n\subset W$ for all $n$ large.\\
If $x^n\in A^n$
so that $x^n$ converges to some $x\in M$, 
then $x\in H$.\\
If $w \in W^{uu}_{X^n}(x^n,\epsilon)$,
the tangent vectors
of $W^{uu}_{X^n}(x^n,\epsilon)$\\ 
in this point are in $E^{u,X^n}_w$.\\
As $x^n\to x$,$W^{uu}_{X^n}(x^n,\epsilon) \to W^{uu}_X(x,\epsilon)$
in the sense of $C^1$ submanifolds \cite{pt}.\\
And $\angle(E^{u,X^n},E^u) \longrightarrow 0$, if $n\to\infty$\cite{m}.\\
\end{tabular}\\
\label{e3}
\end{equation}

Thus, we fix an open set
$U\subset
W^{uu}_X(x,\epsilon)$ containing the point $x$.
By (\ref{e1}), it follows that the periodic orbits 
of $X^n$ in $A^n$ are dense in $A^n$ [Closing-Anosov lemma].
Then we can suppose that for all point $x_n$, $x_n \in O_n \subset A^n$ 
and by Lemma \ref{l1'}, as $M(X)\cap Sing(X) \neq \emptyset$ 
and since $M(X)$ is transitive set, we have that
there are $q\in U$, $0<\delta_1<\delta_2<\frac{\delta}{2}$ and $T>0$ such that
$X_T(q)\in B_{\delta_1}(Sing(X)).$

Thus, there is
an open set $V_q$ containing $q$ such that
$X_T(V_q)\subset B_{\delta_1}(Sing(X)).$\cite[Tubular Flow Box Theorem]{dmp}

As $X^n\to X$ we have
\begin{equation}
\label{intersection}
X^n_T(V_q)\subset B_{\delta_2}(Sing(X))
\end{equation}
for all $n$ large (see Figure \ref{f.3}).

%%===============================================================FIGURE 3=============================================================================
\begin{figure}[htv] 
\begin{center}
\input{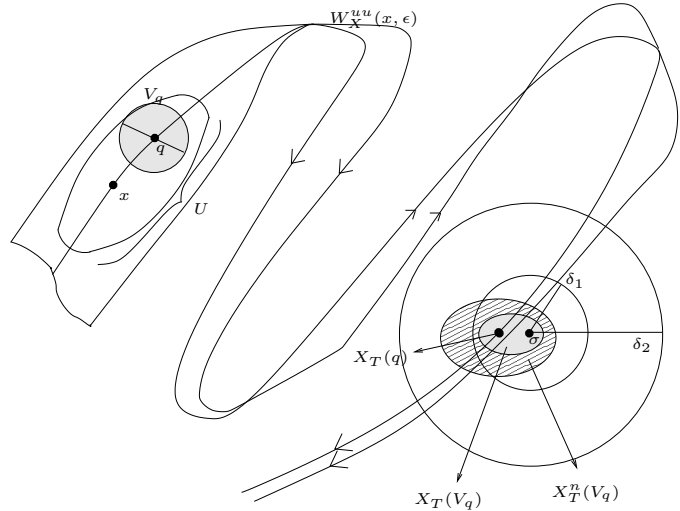} 
\caption{\label{f.3} Tubular Flow Box Theorem for $X_T(V_q)$.}
\end{center}
\end{figure}
%%======================================================================================================================================================

However,
$W^{uu}_{X^n}(x^n,\epsilon)
\to W^{uu}_X(x,\epsilon)
$, $q\in U\subset W^{uu}_X(x,\epsilon)$,
$q\in V_q$ and $V_q$ is open.
So, $W^{uu}_{X^n}(x^n,\epsilon)
\cap V_q\neq\emptyset$ for all $n$ large.
Applying (\ref{intersection})
to $X^n$ for $n$ large we have
$$
X^n_T(W^{uu}_{X^n}(x^n,\epsilon))\cap
B_{\delta_2}(Sing(X))\neq\emptyset.
$$
As $W^{uu}_{X^n}(x^n,\epsilon)\subset W^u_{X^n}(x^n)$
the invariance of $ W^u_{X^n}(x^n)$
implies
$$
W^u_{X^n}(x^n)\cap B_{\delta/2}(Sing(X))
\neq\emptyset.
$$
Observe that $W^u_{X^n}(x^n)\subset A^n$
since
$x^n\in A^n$ and $A^n$ is an attractor.
We conclude that
$$
A^n\cap B_{\delta}(Sing(X))\neq\emptyset.
$$
This contradicts
(\ref{*}) and
the proof follows.
\end{proof}

%%===========================================================PROOF OF THEOREM===============================================================

\noindent
{\flushleft{\bf Proof of Theorem \ref{thB}: }}
By the lemma \ref{Ladilla4}, 
exists $\sigma$, such that $\sigma\in M(X)$ and

$$\sigma\in Sing(X)\bigcap Cl\left(\bigcup_{n\in I\!\! N}A^n\right).$$ 

By Theorem \ref{t1} we have that
$\sigma$ is Lorenz-like
and satisfies
$$
M(X)\cap W^{ss}_X(\sigma)=\{\sigma\}.
$$

By Proposition \ref{Ladilla1}, we can choose 
$\Sigma^t,\Sigma^b$, singular-cross section
for $\sigma$ and $M(X)$ such that
$$
M(X)\cap\left(\partial^h\Sigma^t\cup\partial^h\Sigma^b\right)
=\emptyset.
$$

As $X^n\to X$ we have that
$\Sigma^t,\Sigma^b$ is 
singular-cross section of $X^n$ too, thus
we can assume that $\sigma(X^n)=\sigma$
and $l^t\cup l^b\subset W^s_{X^n}(\sigma)$
for all $n$. [Implicit function theorem].
We have that the splitting
$E^s\oplus E^c$ persists by small
perturbations of $X$ \cite{hps}.

We have that the splitting
$E^s\oplus E^c$ persists by small
perturbations of $X$ \cite{hps}. The dominance condition
[Definition \ref{d2}-(2)]
together with \cite[Proposition 2.2]{d}
imply that for $*=t,b$ one has
$$
T_x\Sigma^*\cap \left(E^s_x\oplus E^X_x\right)
=T_xl^*,
$$
for all $x\in l^*$.

Denote by $\angle(E,F)$ the angle between
two linear subspaces.
The last equality implies
that there is $\rho>0$ such that
$$
\angle(T_x\Sigma^*\cap E_x^c,T_xl^*)>\rho,
$$
for all $x\in l^*$
($*=t,b$).
But $E^{c,n}\to E^c$ as $n\to\infty$.
So for all $n$ large we have
\begin{equation}
\label{angle}
\angle(T_x\Sigma^*\cap E_x^{c,n},T_xl^*)>\frac{\rho}{2},
\end{equation}
for all $x\in l^*$
($*=t,b$).
	
As in the previous section we fix a coordinate system $(x,y)=(x^*,y^*)$
in $\Sigma^*$ such that
$$
\Sigma^*= B^u[0,1] \times B^{ss}[0,1],
\,\,\,\,\,\,\,\,l^*=\{0\}\times B^{ss}[0,1]
$$
with respect to $(x,y)$.

Denote by $\Pi^*:\Sigma^*\to B^u[0,1]$ the projection, where $\Pi^*(x,y)=x.$

As before, given $\Delta>0$ we define
$\Sigma^{*,\Delta}= B^u[0,\Delta] \times B^{ss}[0,1].$

\begin{rk}
\label{Ladilla5}
The continuity of $E^{c,n}$ and (\ref{angle}) imply
that $\exists \Delta_0>0$ such that $\forall n$ large
the line 
$F^n$ is {\em transverse to $\Pi^*$}. By this we mean
that $F^n(z)$ is {\em not tangent
to the curves $(\Pi^*)^{-1}(c)$, $\forall c\in B^u[0,\Delta_0]$}.
\end{rk}

We define the line field $F^n$ in $\Sigma^{*,\Delta_0}$ by
$$
F^n_x=
T_x\Sigma^*\cap E^{c,n}_x, \,\,\,\,x\in \Sigma^{*,\Delta_0}.
$$

Now recall that $A^n$ is a hyperbolic attractor of type saddle
of $X^n$ for all $n$ (see (\ref{e1})) and 
that the periodic orbits of $X^n$ in $A^n$ are dense in $A^n$ (\cite{pt}).
Then, as $\sigma\in Cl\left(\cup_{n\in I\!\! N}A^n\right)$, there is a periodic orbit
sequence $O_n\in A^n$ accumulating on $\sigma$. It follows from
Lemma \ref{Ladilla3} applied to $Y^n=X^n$ that there is
$n_0\in I\!\! N$ such that either
$$
O_{n_0}\cap int(\Sigma^{t,\Delta_0})\neq\emptyset\,\,\,\,
\mbox{or}\,\,\,\,
O_{n_0}\cap int(\Sigma^{b,\Delta_0})\neq\emptyset.
$$
Because $O_{n_0}\subset A_{n_0}$
we conclude that either
$$
A^{n_0}\cap int(\Sigma^{t,\Delta_0})\neq\emptyset\,\,\,\,
\mbox{or}\,\,\,\,
A^{n_0}\cap int(\Sigma^{b,\Delta_0})\neq\emptyset.
$$

We shall assume that $A^{n_0}\cap int(\Sigma^{t,\Delta_0})\neq\emptyset$
(Analogous proof for the case $*=b$).
Note that $\partial^h \Sigma^{t,\Delta_0}\subset \partial^h
\Sigma^t$ by definition.
Then, by
Proposition \ref{Ladilla1} one has
$$
A\cap \partial^h\Sigma^{t,\Delta_0}=\emptyset.
$$

As $A^{n_0}$ and $\Sigma^{t,\Delta_0}$ are compact non-empty sets, 
$A^{n_0}\cap \Sigma^{t,\Delta_0}$ is a compact non-empty subset
of $\Sigma^{t,\Delta_0}$,  
hence there is $p\in \Sigma^{t,\Delta_0} \cap A^{n_0}$ such that

$$ dist(\Pi^t(\Sigma^{t,\Delta_0}\cap A^{n_0}),0)=dist(\Pi^t(p),0),$$

where $dist$ denotes the distance in $B^u[0,\Delta_0]$. 
Note that $dist(\Pi^t(p),0)$ is the minimum distance of
$\Pi^t(\Sigma^{t,\Delta_0} \cap A^{n_0})$ to $0$ in $B^u[0,\Delta_0]$.\\

As $p\in A^{n_0}$, we have that $W^u_{X^{n_0}}(p)$
is a well defined submanifold, 
since that $A^{n_0}$ is hyperbolic (\ref{e1}), and
$dim(E^{c})=dim(E^{c,n_0})$(\ref{e2}).\\

By domination [Definition \ref{d2}-(2)], 
$T_z(W^u_{X^{n_0}}(p))=E^{c,n_0}_z$ 
for every $z\in W^u_{X^{n_0}}(p)$. Thus 
$dim(W^u_{X^{n_0}}(p))=(n-s-1)$ (\ref{p1}). Hence,
$$
T_z(W^u_{X^{n_0}}(p))\cap T_z\Sigma^{t,\Delta_0}=
E^{c,n_0}_z\cap T_z\Sigma^{t,\Delta_0}=F^{n_0}_z
$$
for every $z\in W^u_{X^{n_0}}(p)\cap \Sigma^{t,\Delta_0}$.

This show that $W^u_{X^{n_0}}(p)\cap \Sigma^{t,\Delta_0}$ is transversal,
then we have that
$W^u_{X^{n_0}}(p)\cap \Sigma^{t,\Delta_0}$ contains some compact manifold. 
We denote this compact manifold by $K^{n_0}$. Note that $p \in K^{n_0}$ 
[See (\ref{f.4})].
As $K^{n_0} \subset W^u_{X^{n_0}}(p)\cap \Sigma^{t,\Delta_0}$, this
implies that $K^{n_0}$ is tangent to $F^{n_0}$.  

\begin{rk}
\label{p3}
As $dim(E^{c,n_0})=dim(W^u_{X^{n_0}}(p))=(n-s-1)$, by construction
we have that $dim(B^u[0,\Delta_0])=(n-s-1)$.\\  
As $W^u_{X^{n_0}}(p)$, $\Sigma^{t,\Delta_0}$ are submanifolds of $M$, 
$W^u_{X^{n_0}}(p)\cap \Sigma^{t,\Delta_0}$ is transversal and nonempty, 
then $W^u_{X^{n_0}}(p)\cap \Sigma^{t,\Delta_0}$
is a submanifold of $M$ and $dim(W^u_{X^{n_0}}(p))+dim(\Sigma^{t,\Delta_0}) \geq n$.
\end{rk}

%%===============================================================FIGURE 4=============================================================================
\begin{figure}[htv] 
\begin{center}
\input{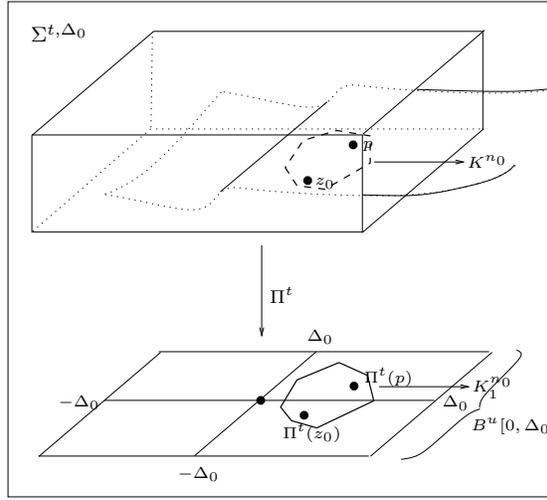} 
\caption{\label{f.4} The projection $\Pi^t(K^{n_0})=K^{n_0}_1$.}
\end{center}
\end{figure}
%%======================================================================================================================================================

Then, since $F^{n_0}$ is {\em transverse}
to $\Pi^t$, we have that
$K^{n_0}$ is {\em transverse} to $\Pi^t$
(i.e. $K^{n_0}$ is transverse to the curves
$(\Pi^t)^{-1}(c)$, for every $c \in B^u[0,\Delta_0]$).
First we denote $\Pi^t(K^{n_0})=K^{n_0}_1$ 
the image of $K^{n_0}$ by the projection $\Pi^t$ in $B^u[0,\Delta_0]$. 
Note that $K^{n_0}_1\subset B^u[0,\Delta_0]$ and 
$\Pi^t(p)\in int(K^{n_0}_1)$.

As $dim(K^{n_0}_1)=dim(B^u[0,\Delta_0])$ (By \ref{p3}),
for this reason, we have that exists  
$z_0\in K^{n_0}$ such that
$$
dist(\Pi^t(z_0),0)<dist(\Pi^t(p),0).
$$

Since that $A^{n_0}$ is an attractor of $X^{n_0}$(\ref{e1}),
we have that $K^{n_0}\subset \Sigma^{t,\Delta_0} \cap A^{n_0}$,
thus $p\in A^{n_0}$
and $K^{n_0}\subset W^u_{X^{n_0}}(p)$.

As $A^{n_0}\cap \partial^h\Sigma^{t,\Delta_0}=\emptyset$
[Proposition \ref{Ladilla1}] and 
$dim(K^{n_0}_1)=dim(B^u[0,\Delta_0])$ (By Remark \ref{p3}),
we conclude that 
$$
dist(\Pi^t(\Sigma^{t,\Delta_0}\cap A^{n_0}),0)=0.
$$
As $A^{n_0}$ is closed, this last equality
implies
$$
A^{n_0}\cap l^t\neq\emptyset.
$$
Since $l^t\subset W^s_{X^{n_0}}(\sigma)$
and $A^{n_0}$ is closed invariant set for $X^{n_0}$
we conclude that $\sigma\in A^{n_0}$.
We have proved that $A^{n_0}$ contains a singularity of $X^{n_0}$.
But $A^{n_0}$ is a hyperbolic attractor of $X^{n_0}$
by the Property (\ref{e1}). Henceforth $A^{n_0}=\{\sigma\}$.
This contradicts the Property (\ref{e1})
and the proof follows.
\qed
%%================================================================================================================00
%%=================================================================================================================00
%%=================================================================================================================00

%%================================================================================================================00
%%=================================================THE ARTICLE IS FINISH =============================================================00
%%=================================================================================================================00

\medskip 

\flushleft
A. M. L\'opez B\\
Instituto de Matem\'atica, Universidade Federal do Rio de Janeiro\\
Rio de Janeiro, Brazil\\
E-mail: barragan@im.ufrj.br

\end{document}